\documentclass[12pt, twoside]{article}
\usepackage{amsmath,amsthm,amssymb}
\usepackage{times}
\usepackage{enumerate}
\usepackage{amsmath,amsfonts}
\usepackage{epsfig}
\usepackage{amsmath}
\usepackage{amssymb}
\usepackage{amscd}
\usepackage{graphicx}
\usepackage{pstricks}
\usepackage{paralist}
\usepackage{slashed}
\usepackage{amsthm}

\def\titlerunning#1{\gdef\titrun{#1}}
\makeatletter
\def\author#1{\gdef\autrun{\def\and{\unskip, }#1}\gdef\@author{#1}}
\def\address#1{{\def\and{\\\hspace*{18pt}}\renewcommand{\thefootnote}{}%
\footnote {#1}}%
\markboth{\autrun}{\titrun}}
\makeatother
\def\email#1{e-mail: #1}
\def\subjclass#1{{\renewcommand{\thefootnote}{}%
\footnote{\emph{Mathematics Subject Classification (2010):} #1}}}
\def\keywords#1{\par\medskip
\noindent\textbf{Keywords.} #1}

\newtheorem{thm}{Theorem}[section]

\newtheorem{lem}[thm]{Lemma}

\newtheorem{prop}[thm]{Proposition}

\theoremstyle{definition}

\newtheorem{rem}[thm]{Remark}

\numberwithin{equation}{section}

\input xy
\xyoption{all}

\def \L {\mathcal{L}}

\def \b {\beta}

\def \de {\delta}
\def \De {\Delta}
\def \la {\lambda}

\def\w {\omega}
\def\Om{\Omega}

\def\pa{\partial}
\def\na {\nabla}
\def\Ga{\Gamma}
\begin{document}

\baselineskip=17pt
\titlerunning{Instantons on Cylindrical Manifolds}
\title{Instantons on Cylindrical Manifolds}
\author{Teng Huang}
\date{}
\maketitle
\address{T. Huang (corresponding author): Department of Mathematics, University of Science and Technology of China,
Hefei, Anhui 230026, PR China; \email{oula143@mail.ustc.edu.cn}}


\begin{abstract}
We consider an instanton,\ $\textbf{A}$\ ,\ with $L^{2}$-curvature $F_{\textbf{A}}$ on the cylindrical manifold $Z=\mathbf{R}\times M$,\ where $M$ is a closed Riemannian $n$-manifold,\ $n\geq 4$.\ We assume $M$ admits a $3$-form $P$ and a $4$-form $Q$ satisfy $dP=4Q$ and $d\ast_{M}{Q}=(n-3)\ast_{M}P$.\ Manifolds with these forms include nearly K\"{a}hler 6-manifolds and nearly parallel $G_{2}$-manifolds in dimension 7.\ Then we can prove that the instanton must be a flat connection.
\end{abstract}

\section{Introduction}

Let $X$ be an $(n+1)$-dimensional Riemannian manifold,\ $G$ be a compact Lie group and $E$ be a principal $G$-bundle on $X$.\ Let $A$ denote a connection on $E$ with the curvature $F_{A}$.\ The instanton equation on $X$ can be introduced as follows.\ Assume there is a $4$-form $Q$ on $X$.\ Then an $(n-3)$-form $\ast{Q}$ exists,\ where $\ast$ is the Hodge operator on $X$.\ A connection,\ $A$\ ,\ is called an anti-self-dual instanton,\ when it satisfies the instanton equation
\begin{equation}\label{1.1}
\ast F_{A}+\ast{Q}\wedge F_{A}=0
\end{equation}
When $n+1>4$, these equations can be defined on the manifold $X$ with a special holonomy group,\ i.e.\ the holonomy group $G$ of the Levi-Civita connection on the tangent bundle $TX$ is a subgroup of the group $SO(n+1)$.\ Each solution of equation(\ref{1.1}) satisfies the Yang-Mills equation.\ The instanton equation (\ref{1.1}) is also well-defined on a manifold $X$ with non-integrable $G$-structures,\ but equation (\ref{1.1}) implies the Yang-Mills equation will have torsion.

Instantons on the higher dimension,\ proposed in \cite{CDFJ} and studied in \cite{RRC,DT,DS,G.Tian,RSW},\ are important both in mathematics \cite{DS,G.Tian} and string theory \cite{GSW}.\ In this paper,\ we consider the cylinder manifold $Z=\mathbf{R}\times{M}$ with metric
$$g_{Z}=dt^{2}+g_{M}$$
where $M$ is a compact Riemannian manifold.\ We assume $M$ admits a $3$-form $P$ and a $4$-form $Q$ satisfying
\begin{eqnarray}\label{1.2}
&dP=4Q \\ \label{1.3}
&d\ast_{M}{Q}=(n-3)\ast_{M}{P}.
\end{eqnarray}
On $Z$,the $4$-form \cite{HN,ID} can be defined as
$$Q_{Z}=dt\wedge P+Q.$$
Then the instanton equation on the cylinder manifold $Z$ is
\begin{equation}\label{1.4}
\ast{F}_{\textbf{A}}+\ast{Q}_{Z}\wedge{F}_{\textbf{A}}=0
\end{equation}
\begin{rem}
Manifolds with $P$ and $Q$  satisfying equations (\ref{1.2}),\ (\ref{1.3}) include nearly K\"{a}hler $6$-manifolds and nearly parallel $G_{2}$-manifolds.

(1)M is a nearly K\"{a}hler 6-manifold.\ It is defined as a manifold with a $2$-form $\w$ and a $3$-form $P$ such that
$$d\w=3\ast_{M}P\ and\ dP=2\w\wedge\w=:4Q$$
For a local orthonormal co-frame $\{e^{a}\}$ on $M$ one can choose
$$\w=e^{12}+e^{34}+e^{56}\ and\ P=e^{135}+e^{164}-e^{236}-e^{245},$$
where $a=1,\ldots,6,$\ $e^{a_{1}\ldots a_{l}}=e^{1}\wedge\ldots e^{l},$\ and
$$\ast_{M}P=e^{145}+e^{235}+e^{136}-e^{246},\ Q=e^{1234}+e^{1256}+e^{3456}.$$
Here $\ast_{M}$ denotes the $\ast$-operator on $M$.

(2)M is a nearly parallel $G_{2}$ manifold.\ It is defined as a manifold with a $3$-form $P$ (a $G_{2}$ structure \cite{Bryant}) preserved by the $G_{2}\subset SO(7)$ such that
$$dP=\gamma\ast_{M}P$$
for some constant $\gamma\in\mathbf{R}$.\ For a local orthonormal co-frame $e^{a}$,\ $a=1,\ldots,7,$ on $M$ one can choose
$$P=e^{123}+e^{145}-e^{167}+e^{246}+e^{257}+e^{347}-e^{356}$$
and therefore
$$\ast_{M}P=:Q=e^{4567}+e^{2367}-e^{2345}+e^{1357}+e^{1346}+e^{1256}-e^{1247}.$$
It is easy to check $dP=4Q$.
\end{rem}
Constructions of solutions of the instanton equations on cylinders over nearly K\"{a}hler $6$-manifolds and nearly parallel $G_{2}$ manifold were considered in \cite{BILL,HILP,ID,ILPR}.\ In \cite{ILPR} section 4, the authors confirm that the standard Yang-Mills functional is infinite on their solutions.\ In this paper,\ we assume the instanton $\textbf{A}$ has  $L^{2}$-bounded curvature $F_{\textbf{A}}$.\ Then we have the following theorem.
\begin{thm}(Main theorem)
Let $Z=\mathbf{R}\times M$,\ here $M$ is a closed Riemannian n-manifold,\ $n\geq 4$,\ which admits a smooth $3$-form $P$ and a smooth $4$-form $Q$ satisfying equations (\ref{1.2}) and (\ref{1.3}).\ Let $\textbf{A}$ be a instanton over $Z$.\ Assume that the curvature $F_{\textbf{A}}\in L^{2}(Z)$ i.e.
$$\int_{Z}\langle F_{\textbf{A}}\wedge\ast F_{\textbf{A}}\rangle<+\infty$$
Then the instanton is a flat connection.
\end{thm}
\section{Esitimation of Curvature of Yang-Mills connection with torsion}

Let $Q$ be a smooth $4$-form on $n$-dimensional manifold $X$.\ Let $A$ be an anti-self-dual instanton which satisfies the instanton equation (\ref{1.1}).\ Taking the exterior derivative of (\ref{1.1}) and using the Bianchi identity,\ we obtain
\begin{equation}\label{2.16}
d_{A}\ast{F_{A}}+\ast\mathcal{H}\wedge F_{A}=0,
\end{equation}
where the $3$-form $\mathcal{H}$  is defined by
\begin{equation}\label{2.17}
\ast\mathcal{H}=d(\ast{Q}).
\end{equation}
The second-order equation (\ref{2.16}) differs from the standard Yang-Mills equation by the last term involving a $3$-from $\mathcal{H}$.\ This torsion term naturally appears in string-theory compactifications with fluxes \cite{BKLS,DK,Gr}.\ For the case $d(\ast{Q})=0$,\ the torsion term vanishes and the instanton equation (\ref{1.1}) imply the Yang-Mills equation.\ The latter also holds true when the instanton solution $A$ satisfies $d(\ast{Q})\wedge F_{A}=0$ as well,\ like the cases,\ on nearly K\"{a}hler $6$-manifolds,\ nearly parallel $G_{2}$-manifolds and Sasakian manifolds \cite{HN}.

In section 4.2 of \cite{BILL} or in section 2.1 of \cite{HILP},\ they online that in the instanton does not extremize the standard Yang-Mills functional in the torsionful case.\ Instead,\ they add a add a Chern-Simons-type term to get the following functional:
\begin{equation}\label{2.18}
S(A)=-\int_{X}Tr\big{(}F_{A}\wedge\ast F_{A}+F_{A}\wedge F_{A}\wedge\ast{Q}\big{)},
\end{equation}
This is the right functional which produces the correct Yang-Mills equation with torsion.\ And the instanton equations (\ref{1.1}) can be derived from this action using a  Bogomolny argument.\ In the case of a closed form $\ast{Q}$,\ the second term in (\ref{2.18}) is topological invariant and the torsion (\ref{2.17}) disappears from (\ref{2.16}).

In this section,\ we will derive monotonicity formula for Yang-Mills connection with torsion (\ref{2.16}).\ Its proof follows Tian's arguments about pure Yang-Mills connection in \cite{G.Tian} with some modifications.

Let $X$ be a compact Riemannian $n$-manifold with metric $g$ and $E$ is a vector bundle over $X$ with compact structure group $G$.\ For any connection $A$ of $E$,\ its curvature form $F_{A}$ takes value in $Lie(G)$.\ The norm of $F_{A}$ at any $p\in X$ is given by
$$|F_{A}|^{2}=\sum_{i,j=1}^{n}\langle F_{A}(e_{i},e_{j}),F_{A}(e_{i},e_{j})\rangle,$$
where $\{e_{i}\}$ is any orthonormal basis of $T_{p}X$,\ and $\langle\cdot,\cdot\rangle$ is the Killing form of $Lie(G)$.

As in \cite{G.Tian},\ we consider a one-parameter family of diffeomorphisms $\{\psi_{t}\}_{|t|<\infty}$ of $X$ with $\psi_{0}=id_{X}$.\ We fix a connection $A_{0}$,\ and denote by its derivative $D$.\ Then for any connection $A$,\ we can define a one-parameter family $\{A_{t}\}$ in the following way.\ Let $\tau_{t}^{0}$ be the parallel transport on $E$ associated to $A_{0}$ along the path $\psi_{s}(x)_{0<s<t}$,\ where $x\in X$.\ More precisely,\ for any $u\in E_{x}$ over $x\in X$,\ let $\tau_{s}^{0}(u)$ be the section of $E$ over the path $\psi_{s}(x)_{0<s<t}$ such that
$$D_{\frac{\pa}{\pa s}}\tau_{s}^{0}(u)=0,\ \tau_{0}^{0}(u)=u.$$
We define a family of connections $A_{t}:=\psi_{t}^{\ast}(A)$ by defining its covariant derivative as
$$D_{\nu}^{t}s=(\tau^{0}_{t})^{-1}\big{(}D_{d\psi_{t}(\nu)}(\tau^{0}_{t}(s))\big{)}$$
for any $\nu\in TX,\ s\in\Gamma(X,E)$.\ Then the curvature of $A_{t}$ is written as
$$F_{A_{t}}(X_{1},X_{2})=(\tau_{t}^{0})^{-1}\cdot F_{A}(d\psi_{t}(X_{1}),d\psi_{t}(X_{2}))\cdot\tau_{t}^{0}.$$
It follows that
$$\int_{X}|F_{A}|^{2}=\int_{X}\sum_{i,j=1}^{n}|F_{A}(d\psi_{t}(e_{i}),d\psi_{t}(e_{j}))|^{2}(\psi_{t}(x))dV_{g}.$$
where $dV_{g}$ denotes the volume form of $g$,\ and $\{e_{i}\}$ is any local orthonormal basis of $TX$.\ By changing variables,\ we obtain
$$\int_{X}|F_{A}|^{2}
=\int_{X}\sum_{i,j=1}^{n}|F_{A}(d\psi_{t}(e_{i}(\psi_{t}^{-1}(x))),d\psi_{t}(e_{j}(\psi_{t}^{-1}(x)))|^{2}Jac(\psi_{t}^{-1})dV_{g}.$$
Let $\nu$ be the vector field $\frac{\pa\psi_{t}}{\pa t}|_{t=0}$ on $X$.\ Then we deduce from the above that
\begin{equation}\label{2.5}
\begin{split}
\frac{d}{dt}YM(A_{t})|_{t=0}&=\int_{X}\langle i_{\nu}F_{A},d_{A}^{\ast}F_{A}\rangle\\
&=\int_{X}Tr\big{(}i_{\nu}F_{A}\wedge F_{A}\wedge (d\ast Q)\big{)}\\
&=\int_{X}\big{(}|F_{A}|^{2}div\nu+4\sum_{i,j=1}^{n}\langle F_{A}([\nu,e_{i}],e_{j}),F_{A}(e_{i},e_{j})\rangle\big{)}dV_{g}\\
&=\int_{X}\big{(}|F_{A}|^{2}div\nu-4\sum_{i,j=1}^{n}\langle F_{A}(\na_{e_{i}}\nu,e_{j}),F_{A}(e_{i},e_{j})\rangle\big{)}dV_{g}\\
\end{split}
\end{equation}
Fix any $p\in X$,\ let $r_{p}$ be a positive number with following properties:\ there are normal coordinates $x_{1},\ldots,x_{n}$ in the geodesic ball $B_{r_{p}}(p)$ of $(X,g)$,\ such that $p=(0,\ldots,0)$ and for some constant $c(p)$,
$$|g_{ij}-\de_{ij}|\leq c(p)(|x_{1}|^{2}+\ldots+|x_{n}|^{2}),$$
$$|dg_{ij}|\leq c(p)\sqrt{|x_{1}|^{2}+\ldots+|x_{n}|^{2}},$$
where $$g_{ij}=g\big{(}\frac{\pa}{\pa x_{i}},\frac{\pa}{\pa x_{j}}\big{)}.$$
Let $r(x):=\sqrt{x_{1}^{2}+\ldots+x_{n}^{2}}$ be the distance function from $p$.\ Define $\nu(x)=\xi(r)r\frac{\pa}{\pa r}$,\ where $\xi$ is some smooth function with compact support in $B_{r_{p}}(p)$.\ Let $\{e_{1},\ldots,e_{n}\}$ be any orthonormal basis near $p$ such that $e_{1}=\frac{\pa}{\pa r}$.\ Since $x_{1},\ldots,x_{n}$ are normal coordinates,\ we have $\na_{\frac{\pa}{\pa r}}\frac{\pa}{\pa r}=0$.\ It follows that
\begin{equation}\label{2.6}
\na_{\frac{\pa}{\pa r}}\nu=(\xi r)'\frac{\pa}{\pa r}=(\xi'r+\xi)\frac{\pa}{\pa r}.
\end{equation}
Moreover,\ for $i\geq2$,
\begin{equation}\label{2.7}
\na_{e_{i}}\nu=\xi r(\na_{e_{i}}\frac{\pa}{\pa r})=\xi\sum_{j=1}^{n}b_{ij}e_{j}.
\end{equation}
where $|b_{ij}-\de_{ij}|=O(1)c(p)r^{2}$.\ Applying (\ref{2.6}) and (\ref{2.7}) to the variation formula (\ref{2.5}),\ we obtain
\begin{equation}\label{2.8}
\begin{split}
\int_{X}\langle i_{\nu}F_{A},d_{A}^{\ast}F_{A}\rangle
&=\int_{X}\big{(}|F_{A}|^{2}(\xi'r+(n-4)\xi+O(1)c(p)r^{2}\xi)dV_{g}\\
&-4\int_{X}(\xi'r|\frac{\pa}{\pa r}\lrcorner F_{A}|^{2})dV_{g}.\\
\end{split}
\end{equation}
where $\frac{\pa}{\pa r}\lrcorner F_{A}=F_{A}(\frac{\pa}{\pa r},\cdot)$.

We choose,\ for any $\tau$ small enough.\ $\xi(r)=\xi_{\tau}(r)=\eta(\frac{r}{\tau})$,\ where $\eta$ is smooth and satisfies:\ $\eta(r)=1$ for $r\in[0,1]$,\ $\eta(r)=0$ for $r\in[1+\epsilon,\infty)$, $\epsilon<0$ and $\eta'(r)\leq0$.\ Then
\begin{equation}\label{2.9}
\tau\frac{\pa}{\pa\tau}(\xi_{\tau}(r))=-r\xi'_{\tau}(r).
\end{equation}
Plugging this into (\ref{2.8}),\ we obtain
\begin{equation}
\begin{split}
\int_{X}\langle i_{\nu}F_{A},d_{A}^{\ast}F_{A}\rangle
&=\tau\frac{\pa}{\pa\tau}\big{(}\int_{X}\xi_{\tau}|F_{A}|^{2}dV_{g}\big{)}\\
&+\big{(}(4-n)+O(1)c(p)\tau^{2}\big{)}\int_{X}\xi_{\tau}|F_{A}|^{2}dV_{g}\\
&-4\tau\frac{\pa}{\pa\tau}\big{(}\int_{X}\xi_{\tau}|\frac{\pa}{\pa r}\lrcorner F_{A}|^{2}dV_{g}\big{)}\\
\end{split}
\end{equation}
Choose a nonngeative number $a\geq O(1)c(p)r_{p}+\max_{x\in X}|d(\ast{Q})|(x)$.\ Then we deduce from the above
\begin{equation}
\begin{split}
&\quad\frac{\pa}{\pa\tau}\big{(}\tau^{4-n}e^{a\tau}\int_{X}\xi_{\tau}|F_{A}|^{2}dV_{g}{)}\\
&=\tau^{4-n}e^{a\tau}\bigg{(}4\frac{\pa}{\pa\tau}\big{(}\int_{X}\xi_{\tau}|\frac{\pa}{\pa r}\lrcorner F_{A}|^{2}dV_{g}\big{)}+(-O(1)c(p)\tau+a)\int_{X}\xi_{\tau}|F_{A}|^{2}dV_{g}\bigg{)}\\
&-e^{a\tau}\tau^{3-n}\int_{X}Tr(i_{\nu}F_{A}\wedge F_{A}\wedge d(\ast Q))\\
\end{split}
\end{equation}
We have the fact:
\begin{equation}
\begin{split}
|\int_{X}Tr(i_{\nu}F_{A}\wedge F_{A}\wedge d(\ast Q))|&\leq\int_{X}|\nu|\cdot|F_{A}|^{2}\cdot\max_{x\in X}|d(\ast{Q})|\\
&\leq\max_{x\in X}|d(\ast{Q})|\int_{B_{\tau(1+\epsilon)}(p)}\tau(1+\epsilon)|F_{A}|^{2}dV_{g}\\
\end{split}
\end{equation}
Then,\ by integrating on $\tau$ and letting $\epsilon$ tend to zero,\ we have already proved:
\begin{thm}\label{2.12}
Let $r_{p},\ c(p)$ and $a$ be as above.\ Then for any $0<\sigma<\rho<r_{p}$,\ we have
\begin{equation}
\begin{split}
&\quad\rho^{4-n}e^{a\rho}\int_{B_{\rho}(p)}|F_{A}|^{2}dV_{g}-\sigma^{4-n}e^{a\sigma}\int_{B_{\sigma}(p)}|F_{A}|^{2}dV_{g}\\
&\geq4\int_{B_{\rho}(p)\backslash B_{\sigma}(p)}r^{4-n}e^{ar}|\frac{\pa}{\pa r}\lrcorner F_{A}|^{2}dV_{g}\\
&+\int_{\sigma}^{\rho}\big{(}e^{a\tau}\tau^{4-n}(a-\max_{x\in X}|d(\ast{Q})|-O(1)c_{p}\tau)\int_{B_{\tau}(p)}|F_{A}|^{2}dV_{g}\big{)}d\tau\\
\end{split}
\end{equation}
\end{thm}
Next,\ we prove a mean value inequality about the energy dense $|F_{A}|^{2}$.\ The Bochner-Weitzenb\"{o}ck formula (\cite{BL}, Theorem 3.1) is
$$(d_{A}d_{A}^{\ast}+d_{A}^{\ast}d_{A})F_{A}=\na^{\ast}_{A}\na_{A}F_{A}
+F_{A}\circ(Ric\wedge g+2R)+\mathcal{R}^{A}(F_{A}).$$
Since $A$ is a instanton and the Bianchi identity $d_{A}F_{A}=0$,\ we re-write the left hand
\begin{equation}\nonumber
(d_{A}d_{A}^{\ast}+d_{A}^{\ast}d_{A})F_{A}=d_{A}\big{(}\ast(d(\ast Q)\wedge F_{A})\big{)}
\end{equation}
Hence
$$|d_{A}\big{(}\ast(d(\ast Q)\wedge F_{A})\big{)}|\leq n|\na(d(\ast Q))|F_{A}|+n|d(\ast Q)||\na_{A}F_{A}|$$
due to for all $X_{1},X_{2},\ldots,X_{n-2}\in T_{x}X$,\ where $n=dimX$,
\begin{equation}\nonumber
\begin{split}
&\quad d_{A}^{\ast}\big{(}(d(\ast Q)\wedge F_{A})\big{)}(X_{1},X_{2},\ldots,X_{n-2})\\
&=-\sum_{i}(\na_{A})_{e_{i}}(d(\ast Q)\wedge F_{A})(e_{i},X_{1},\ldots,X_{n-2}),\\
&=-\sum_{i}\big{(}\na_{e_{i}}(d(\ast Q))\wedge F_{A}\big{)}(e_{i},X_{1},\ldots,X_{n-2})\\
&\quad-\sum_{i}\big{(}(d(\ast Q))\wedge (\na_{A})_{e_{i}}F_{A}\big{)}(e_{i},X_{1},\ldots,X_{n-2}),\\
\end{split}
\end{equation}
Then we have
\begin{equation}
\begin{split}
|\langle(d_{A}d_{A}^{\ast}+d_{A}^{\ast}d_{A})F_{A},F_{A}\rangle|
&\leq n|\na(d(\ast Q))|\cdot|F_{A}|^{2}+n|d(\ast Q)|\cdot|\na_{A}F_{A}|\cdot|F_{A}|\\
&\leq C_{1}|F_{A}|^{2}+C_{2}(\varepsilon|\na_{A}F_{A}|^{2}+\frac{1}{\varepsilon}|F_{A}|^{2})\\
\end{split}
\end{equation}
where $\varepsilon$ is a positive constant.

The quadratic $\mathcal{R}^{A}(F_{A})\in\Om^{2}(\mathfrak{g})$ can be expressed with the help of a local orthonormal frame $(e_{1},e_{2},\ldots,e_{n})$ of $TX$ as
$$\mathcal{R}^{A}(F_{A})(X_{1},X_{2})=2\sum_{j=1}^{n}[F_{A}(e_{j},X_{1}),F_{A}(e_{j},X_{2})].$$
The estimate of the Laplacian now follow from
\begin{equation}\nonumber
\begin{split}
-\na^{\ast}\na|F_{A}|^{2}
&=-2|\na_{A}F_{A}|^{2}-2\langle F_{A},\na_{A}^{\ast}\na_{A}F_{A}\rangle\\
&\leq 2\langle F_{A},F_{A}\circ(Ric\wedge g+2R)\rangle+2\langle F_{A},\mathcal{R}^{A}(F_{A})\rangle\\
&+C_{1}|F_{A}|^{2}+C_{2}(\varepsilon|\na_{A}F_{A}|^{2}
+\frac{1}{\varepsilon}|F_{A}|^{2})-2|\na_{A}F_{A}|^{2}\\
\end{split}
\end{equation}
We choose $\varepsilon$ small enough such that $C_{2}\varepsilon<2$,\ then we have
\begin{equation}\nonumber
\De|F_{A}|^{2}\leq C|F_{A}|^{2}+c|F_{A}|^{3}.
\end{equation}
Thus,\ we get
\begin{equation}\label{2.11}
\De|F_{A}|\leq C|F_{A}|+c|F_{A}|^{2}.
\end{equation}
\begin{thm}\label{2.19}
Let $A$ be any Yang-Mills connection with torsion of a $G$-bundle $E$ over $X$.\ Then there exist constants $\varepsilon=\varepsilon(X,n,Q)>0$ and $C=C(X,n)$,\ such that for any $p\in X$ and $p<r_{p}$,\ whenever
$$\rho^{4-n}\int_{B_{\rho}(p)}|F_{A}|^{2}dV_{g}\leq\varepsilon$$
then
\begin{equation}\label{2.15}
|F_{A}|(p)\leq\frac{C}{\rho^{2}}\big{(}\rho^{4-n}\int_{B_{\rho}(p)}|F_{A}|^{2}dV_{g}\big{)}^{\frac{1}{2}}.
\end{equation}
\end{thm}
Our proof here use G.Tian's arguments in \cite{G.Tian} for pure Yang-Mills connection.
\begin{proof}
By scalling,\ we may assume that $\rho=1$.\ Define a function
$$f(r)=(1-2r)^{2}\sup_{x\in B_{r}(p)}|F_{A}|(x),\ r\in[0,\frac{1}{2}].$$
Then $f(r)$ is continuous in $[0,\frac{1}{2}]$ with $f(\frac{1}{2})=0$,\ and $f$ attains its maximum at a certain $r_{0}$ in $[0,\frac{1}{2}]$.

First we claim that $f(r_{0})\leq64$ if $\varepsilon$ is sufficiently small.\ Assume that $f(r_{0})>64$.\ Put $b=\sup_{x\in B_{r_{0}}(p)}|F_{A}|(x)=|F_{A}|(x_{0})$ by taking $\sigma=\frac{1}{4}(1-2r_{0})$,\ we get
\begin{equation}
\begin{split}
\sup_{x\in B_{\sigma}(x_{0})}|F_{A}|&\leq\sup_{x\in B_{r_{0}+\sigma}(p)}|F_{A}|(x)\\
&\leq\frac{(1-2r_{0})^{2}}{(1-2r_{0}-2\sigma)^{2}}\sup_{x\in B_{r_{0}}(p)}|F_{A}|(x)=4b.\\
\end{split}
\end{equation}
Clearly,\ $16\sigma^{2}b\geq64$;\ i.e.,\ $\sigma\sqrt{b}\geq2$.\ Define a scaled metric $\widetilde{g}=bg$.\ Then the norm $|F_{A}|_{\widetilde{g}}$ of $F_{A}$ is equal to $b^{-1}F_{A}$ with respect to $\widetilde{g}$.\ Hence
\begin{equation}\label{2.10}
\sup_{x\in B_{2}(x_{0},\widetilde{g})}|F_{A}|_{\widetilde{g}}\leq4,
\end{equation}
where $B_{2}(x_{0},\widetilde{g})$ denotes the geodesic ball of $\widetilde{g}$ with radius $2$ and centered at $x_{0}$.
Using (\ref{2.10}),\ we deduce from (\ref{2.11}) that in $B_{2}(x_{0},\widetilde{g})$,
\begin{equation}
\De_{\widetilde{g}}|F_{A}|_{\widetilde{g}}\leq(C+4c)|F_{A}|_{\widetilde{g}}.
\end{equation}
Then,\ by using the mean-value theorem,\ we obtain
\begin{equation}\label{2.13}
1=|F_{A}|_{\widetilde{g}}(x_{0})\leq\widetilde{c}\big{(}\int_{B_{1}(x,\widetilde{g})}|F_{A}|_{\widetilde{g}}^{2}dV_{\widetilde{g}}\big{)}^{\frac{1}{2}}.
\end{equation}
where $\widetilde{c}$ is some uniform constant.

However,\ by the monotonicity (Theorem \ref{2.12}),
\begin{equation}\nonumber
\begin{split}
\int_{B_{1}(x_{0},\widetilde{g})}|F_{A}|^{2}_{\widetilde{g}}dV_{\widetilde{g}}&=(\sqrt{b})^{n-4}\int_{B_{\frac{1}{\sqrt{b}}}(x_{0})}|F_{A}|^{2}dV_{g}\\
&\leq(\frac{1}{2})^{4-n}e^{\frac{a}{2}}\int_{B_{\frac{1}{2}}(x_{0})}|F_{A}|^{2}dV_{g}\\
&\leq\varepsilon2^{n-4}e^{\frac{a}{2}}\\
\end{split}
\end{equation}
Combining this with (\ref{2.13}),\ we obtain
$$1\leq\widetilde{c}\varepsilon2^{n-4}e^{\frac{a}{2}}.$$
It is impossible since we can choose  $\varepsilon=\varepsilon(X,n,Q)$ sufficiently small.\ The claim is proved.

Thus,\ we have
$$\sup_{x\in B_{\frac{1}{4}}(p)}|F_{A}|(x)\leq4f(r_{0})\leq256.$$
It follows from this and (\ref{2.11}) with $\widetilde{g}$ replaced by $g$ that for some uniform constant $c'$,
\begin{equation}\label{2.14}
\De_{g}|F_{A}|\leq c'|F_{A}|.
\end{equation}
Then (\ref{2.15}) follows from (\ref{2.14}) and a standard Moser iteration.
\end{proof}

\section{Asymptotic Behavior and Conformal Transformation}

\subsection{Chern-Simons Functional}

The main aim of this section is to get the relationship between gauge theory on an $n$-dimensional manifold $M$ and the gauge theory on the $n+1$-dimensional manifold $Z=\mathbf{R}\times{M}$.\ The main idea is that a connection on $\mathbf{R}\times{M}$ can be regard as one-parameter families of connections on $M$ by local trivialisation.\ Let $t$ be the standard parameter on the factor $\mathbf{R}$ in the $\mathbf{R}\times{M}$ and let $\{x^{j}\}_{j=1}^{n}$ be local coordinates of $M$.\ A connection $\textbf{A}$ over the cylinder $Z$ is given by a local connection matrix
$$\textbf{A}=A_{0}dt+\sum_{i=1}^{n}A_{i}dx^{i}.$$
where $A_{0}$ and $A_{i}$ dependence on all $n+1$ variable $t,x^{1},\ldots,x^{n}$.\ We take $A_{0}=0$ (sometimes called a temporal gauge).\ In this situation, the curvature in a mixed $x_{i}$-plane is given by the simple formula
$$F_{0i}=\frac{\pa A_{i}}{\pa t}.$$
We denote $A=\sum_{i=1}^{n}A_{i}dx^{i}$ and $\dot{A}=\frac{\pa A}{\pa t}$,\ then the curvature is given by
$$F_{\textbf{A}}=F_{A}+dt\wedge\dot{A}.$$
$M$ has a Riemannian metric and $\ast$-operator $\ast_{M}$.\ If $\phi$ is a $1$-form on $M$ then,\ for $\ast$-operator defined on $Z=\mathbf{R}\times M$ with respect to the product metric,\ we have
$$\ast(dt\wedge\phi)=\ast_{M}\phi.$$
Then the instanton equation is equivalent to
\begin{equation}\label{2.1}
\ast_{M}\dot{A}=-\ast_{M}P\wedge F_{A},
\end{equation}
\begin{equation}
\ast_{M}F_{A}=-\dot{A}\wedge\ast_{M}P-\ast_{M}Q\wedge F_{A}.
\end{equation}
Let $E\rightarrow M$ be a vector bundle,\ the space $\mathcal{A}$ is an affine space modelled on $\Om^{1}(\mathfrak{g}_{E})$ so,\ fixing a reference connection $A_{0}\in\mathcal{A}$,\ we have
$$\mathcal{A}=A_{0}+\Om^{1}(\mathfrak{g}_{E}).$$\ We define the Chern-Simons functional by
$$CS(A):=-\int_{M}Tr\big{(}a\wedge{d_{A_{0}}a}+\frac{2}{3}a\wedge a\wedge a\big{)}\wedge\ast_{M}P,$$
fixing $CS(A_{0})=0$.\ This functional is obtained by integrating of the Chern-Simons $1$-form
$$\Gamma(\b)_{A}=\Gamma_{A}(\b_{A})=-2\int_{M}Tr(F_{A}\wedge\b_{A})\wedge\ast_{M}P.$$
We find $CS$ explicitly by integrating $\Ga$ over paths $A(t)=A_{0}+ta$,\ from $A_{0}$ to any $A=A_{0}+a$:
\begin{equation}\nonumber
\begin{split}
CS(A)-CS(A_{0})&=\int_{0}^{1}\Gamma_{A(t)}\big{(}\dot{A}(t)\big{)}dt\\
&=-2\int_{0}^{1}\Big{(}\int_{M}Tr\big{(}(F_{A_{0}}+td_{A_{0}}a+t^{2}a\wedge a)\wedge a\big{)}\wedge\ast_{M} P\Big{)}dt\\
&=-\int_{M}Tr\big{(}d_{A_{0}}a\wedge a+\frac{2}{3}a\wedge a\wedge a\big{)}\wedge\ast_{M}P+C,\\
\end{split}
\end{equation}
where $C=C(A_{0},a)$ is a constant and vanishes if $A_{0}$ is an instanton.\ The co-closed condition $d\ast_{M}{P}=0$ implies that the Chern-Simons $1$-form is closed.\ So it does not depend on the path $A(t)$ \cite{Donaldson,Earp,Earp2}.
Since
$$dTr(d_{A_{0}}a\wedge a+\frac{2}{3}a\wedge a\wedge a)=Tr(F^{2}_{A_{0}+a}-F^{2}_{A_{0}}),$$
we can re-write Chern-Simons functional as
\begin{equation}\label{2.3}
\begin{split}
CS(A)-CS(A_{0})&=-\int_{M}Tr\big{(}d_{A_{0}}a\wedge{a}+\frac{2}{3}a\wedge a\wedge a\big{)}\wedge\ast_{M}P\\
&=-\frac{1}{n-3}\int_{M}Tr(F^{2}_{A}-F^{2}_{A_{0}})\wedge\ast_{M}Q,\\
\end{split}
\end{equation}
the second formula holds because of equation (\ref{1.3}).

\subsection{Asymptotic Behavior}

Let $Z=\mathbf{R}\times{M}$ be an $(n+1)$-manifold and $M$ be an $n$-manifold.\ Let $\textbf{A}$ be an instanton on $Z$ with finite energy,\ i.e.\ $\int_{Z}|F_{\textbf{A}}|^{2}<\infty$.\ We use the Chern-Simons functional to study the decay of instantons over the cylinder manifold.\ We will see that,\ an instanton with $L^{2}(Z)$-bounded curvature can be represented by a connection form which decays exponentially on the tube.

We consider a family of bands $B_{T}=(T-1,T)\times M$ which we identify with the model $B=(0,1)\times M$ by translation.\ So the integrability of $|F_{\textbf{A}}|^{2}$ over the end implies that
$$\int_{(T,T+1)\times M}|F_{\textbf{A}}|^{2}\rightarrow0\quad as\quad T\rightarrow\infty.$$
\begin{prop}\label{3.4}
Let $Z=\mathbf{R}\times M$,\ here $M$ is a closed Riemannian $n$-manifold,\ $n\geq 4$,\ which admits a smooth $3$-form $P$ and a smooth $4$-form $Q$ those satisfy equations (\ref{1.2}) and (\ref{1.3}).\ Let $\textbf{A}$ be a instanton over $Z$,\ then at the end of $Z$ there is a flat connection $\Gamma$ over $M$ such that $\textbf{A}$ converges to $\Gamma$,\ i.e.\ the restriction $\textbf{A}|_{M\times\{T\}}$ converges (modulo gauge equivalence) in $C^{\infty}$ over $M$ as $T\rightarrow\infty$.
\end{prop}
\begin{proof}
We choose $$\rho=\frac{1}{2}Inj\big{(}(T,T+1)\times M,g_{Z}\big{)},$$
where $Inj\big{(}(t,t+1)\times M\big{)}>0$ denotes the injectivity radius of the manifold
$\big{(}(T,T+1)\times M,g_{Z}\big{)}$.\ It's easy to see $\rho$ is not dependent on $t$.
Since $\ast{Q_{Z}}=\ast_{M}P+dt\wedge\ast_{M}Q$,\ we obtain
$$\max_{(x,t)\in Z}|d(\ast{Q_{Z}})|^{2}=\max_{x\in M}\big{(}|d\ast_{M}P|^{2}+|d(\ast_{M}{Q})|^{2}\big{)}<\infty,$$
and
$$\max_{(x,t)\in Z}|\na(d\ast{Q_{Z}})|\leq\max_{x\in M}|\na(d\ast_{M}P)|+\max_{x\in M}|\na(d\ast_{M}{Q})|<\infty.$$
We denote $\varepsilon=\varepsilon(Z,n,Q)$ as the constant in Theorem \ref{2.19}.\ Then there exist $T$ sufficiently large such that $t\geq T$,\ we have
$$\int_{(T,T+1)\times M}|F_{\textbf{A}}|^{2}\leq\varepsilon\rho^{n-3},$$
Then for any point $(t,x)\in (T,T+1)\times M$,\ we have
$$\rho^{3-n}\int_{B_{\rho}(x,t)}|F_{\textbf{A}}|^{2}\leq\varepsilon.$$
From Theorem \ref{2.19},\ we have
$$|F_{\textbf{A}}|(t,x)\leq \frac{C}{\rho^{2}}\big{(}\rho^{3-n}\int_{B_{\rho}(x,t)}|F_{\textbf{A}}|^{2}\big{)}^{\frac{1}{2}}$$
It implies that for any sequence $T_{i}\rightarrow\infty$ there exist a flat connection $\Ga$ over $M$ such that,\ after suitable gauge transformations,
$$\textbf{A}_{T_{i}}\rightarrow\Ga,$$
in $C^{\infty}$ over $M$.
\end{proof}
Under above,\ from (\ref{2.3}) we can write Chern-Simons function as
$$CS(A(T))-CS(A(\infty))=-\frac{1}{n-3}\int_{M}Tr(F_{A}\wedge F_{A})\wedge\ast_{M}Q.$$
\begin{lem}Let $\textbf{A}$ be an instanton with temporal gauge,\ then
\begin{equation}\label{2.2}
\begin{split}
CS(A(T'))-CS(A(T))&=\int_{[T,T']\times{M}}Tr(F_{\textbf{A}}\wedge\ast{F}_{\textbf{A}})\\
&-(n-3)\int_{T}^{T'}\big{(}CS(A(t))-CS(A_{\infty})\big{)}dt\\
\end{split}
\end{equation}
\end{lem}
\begin{proof}
Using the method of previous section,\ we have
$$\frac{d}{dt}CS\big{(}A(t)\big{)}=\Gamma_{A(t)}(\dot{A(t)})$$
Then
\begin{equation}\nonumber
\begin{split}
&\quad CS(A(T'))-CS(A(T))=\int_{T}^{T'}{dCS\big{(}A(t)\big{)}}=\int_{T}^{T'}\Gamma_{A(t)}\big{(}\dot{A}(t)\big{)}dt\\
&=-2\int_{[T,T']\times M}Tr(F_{A(t)}\wedge dt\wedge\dot{A}(t))\wedge\ast_{M}P\\
&=-\int_{[T,T']\times M}Tr(F_{\textbf{A}}\wedge F_{\textbf{A}})\wedge\ast Q_{Z}+\int_{T}^{T'}\Big{(}\int_{M}Tr(F_{A}\wedge F_{A})\wedge\ast_{M}Q\Big{)}dt\\
&=\int_{[T,T']\times M}Tr(F_{\textbf{A}}\wedge\ast F_{\textbf{A}})-(n-3)\int_{T}^{T'}\big{(}CS(A(t))-CS(A(\infty))\big{)}dt\\
\end{split}
\end{equation}
\end{proof}
We set
$$J(T)=\int_{T}^{\infty}\|F_{\textbf{A}}\|_{L^{2}}^{2}=-\int_{[T,\infty)\times M}Tr(F_{\textbf{A}}\wedge\ast F_{\textbf{A}}).$$
On the one hand, we can express $J(T)$ as the integration of $Tr(F_{\textbf{A}}\wedge F_{\textbf{A}})\wedge\ast Q_{Z}$,\ since $\textbf{A}$ is an instanton.
$$J(T)=\int_{[T,\infty)\times M}Tr(F_{\textbf{A}}\wedge F_{\textbf{A}})\wedge\ast Q_{Z}$$
From (\ref{2.2}),\ taking the limit over finite tubes $(T,T')\times M$ with $T'\rightarrow +\infty$ we see that
\begin{equation}\label{3.5}
J(T)=CS(A(T))-CS(A_{\infty})-(n-3)\int_{T}^{\infty}\big{(}CS(A(t))-CS(A(\infty))\big{)}dt
\end{equation}
where $A(T)$ is the connection over $M$ obtain by restriction to $M\times\{T\}$.\ From (\ref{3.5}),\ we can obtain the $T$ derivative of $J$ as
\begin{equation}\label{3.2}
\frac{d}{dT}J(T)=\frac{d}{dT}\big{(}CS(A(T))-CS(A(\infty))\big{)}+(n-3)\big{(}CS(A(T))-CS(A_{\infty})\big{)}
\end{equation}
On the other hand, the $T$ derivative of $J(T)$ can be expressed as minus the integration over $M\times\{T\}$ of the curvature density $|F_{\textbf{A}}|^{2}$,\ and this is exactly the $n$-dimensional curvature density $|F_{A(T)}|^{2}$ plusing the density $|\dot{A}|^{2}$.\ By the relation (\ref{1.2}) and (\ref{1.3}) between the two components of the curvature for an instanton,\ we have
\begin{equation}\nonumber
\begin{split}
\|F_{A(T)}\|^{2}_{L^{2}(M)}&=-\int_{M}Tr(F_{A(T)}\wedge\ast_{M}F_{A(T)})\\
&=-\int_{M}Tr\big{(}F_{A(T)}\wedge(-\dot{A}(T)\wedge\ast_{M}P)\big{)}\\
&\quad+\int_{M}Tr(F_{A(T)}\wedge F_{A(T)})\wedge\ast_{M}Q\\
&=\|\dot{A}\|_{L^{2}(M)}^{2}-(n-3)\big{(}CS(A(T))-CS(A_{\infty})\big{)}\\
\end{split}
\end{equation}
Thus
\begin{eqnarray}\label{3.3}
\frac{d}{dT}J(T)&=-2\|F_{A(T)}\|^{2}_{L^{2}(M)}-(n-3)\big{(}CS(A(T))-CS(A_{\infty})\big{)}\\ \label{3.6}
&=-2\|\dot{A}\|_{L^{2}(M)}^{2}+(n-3)\big{(}CS(A(t))-CS(A_{\infty})\big{)}
\end{eqnarray}
From (\ref{3.2}) and (\ref{3.3}),\ we have
$$\frac{d}{dT}\big{(}CS(A(T))-CS(A_{\infty})\big{)}+2(n-3)\big{(}CS(A(T))-CS(A_{\infty})\big{)}\leq 0$$
From (\ref{3.2}) and (\ref{3.6}),\ we have
$$ \frac{d}{dT}\big{(}CS(A(T))-CS(A_{\infty})\big{)}\leq 0$$
It's easy to see these imply that $\big{(}CS(A(t))-CS(A_{\infty})\big{)}$ is non-negative and decays exponentially,
\begin{equation}\label{3.7}
0\leq\big{(}CS(A(T))-CS(A_{\infty})\big{)}\leq Ce^{-(2n-6)T}
\end{equation}
We introduce a parameter $\delta$ and set
$$L_{\delta}(T):=\int_{T}^{\infty}e^{\delta t}\|F_{\textbf{A}}\|_{L^{2}(M)}^{2}dt$$
\begin{thm}\label{3.8}
Let $\textbf{A}$ be an instanton with $L^{2}$-bounded curvature on $Z=\mathbf{R}\times M$,\ here $M$ is a closed Riemannian n-manifold,\ $n\geq 4$,\ which admits a smooth $3$-form $P$ and a smooth $4$-form $Q$ those satisfy equations (\ref{1.2}) and (\ref{1.3}).\ Then there is a constant $C$,\ such that
$$L_{\delta}(t)\leq Ce^{(\delta-2n+6)t}$$
where $0<\delta<2n-6$.
\end{thm}
\begin{proof}
From (\ref{3.2}),\ we get
$$\|F_{\textbf{A}}\|_{L^{2}(M)}^{2}=-\frac{d}{dt}\big{(}CS(A(t))-CS(A_{\infty})\big{)}-(n-3)\big{(}CS(A(t))-CS(A_{\infty})$$
Then
\begin{equation}\nonumber
\begin{split}
L_{\delta}(T)&=-\int_{T}^{\infty}e^{\delta t}\frac{d}{dt}\big{(}CS(A(t))-CS(A_{\infty})\big{)}\\
&\quad-(n-3)\int_{T}^{\infty}e^{\delta t}\big{(}CS(A(t))-CS(A_{\infty})\big{)}\\
&=-e^{\delta t}\big{(}CS(A(t))-CS(A_{\infty})\big{)}|^{\infty}_{T}+\int_{T}^{\infty}\delta e^{\delta t}\big{(}CS(A(t))-CS(A_{\infty})\big{)}\\
&\quad-(n-3)\int_{T}^{\infty}e^{\delta t}\big{(}CS(A(t))-CS(A_{\infty})\big{)}\\
&\leq e^{\delta T}\big{(}CS(A(T))-CS(A_{\infty})\big{)}+\int_{T}^{\infty}\delta e^{\delta t}\big{(}CS(A(t))-CS(A_{\infty})\big{)}\\
&\leq Ce^{(\delta-2n+6)T}+\int_{T}^{\infty}C\delta e^{(\delta-2n+6)t}dt\\
&=(C+\frac{C\delta}{2n-6-\delta})e^{(\delta-2n+6)T}\\
\end{split}
\end{equation}
\end{proof}

\subsection{Conformal Transformation}

We consider $\bar{Z}=C(M)$,\ where $C(M)$ is a cone over $M$ with metric
$$g_{\bar{Z}}=d r^{2}+r^{2}g_{M}=e^{2t}(dt^{2}+g_{M}),$$
where $r:=e^{t}$. \\  \\
It means that the cone $C(M)$ is conformally equivalent to the cylinder
$$Z=\mathbf{R}\times M$$
with the metric
$$g_{Z}=dt^{2}+g_{M}.$$
Furthermore,\ we can show that the instanton equation on the cone $\bar{Z}=C(M)$ is related with the instanton equation on the cylinder $Z=\mathbf{R}\times M$ ,
$$\bar{\ast}F_{\textbf{A}}+\bar{\ast}Q_{\bar{Z}}\wedge F_{\textbf{A}}=e^{(n-3)t}(\ast F_{\textbf{A}}+\ast Q_{Z}\wedge F_{\textbf{A}})=0,$$
where $dimC(M)=dimZ=n+1$,\ $\bar{\ast}$ is the $\ast$-operator in $C(M)$.And
\begin{equation}
Q_{\bar{Z}}=e^{4t}(dt\wedge P+Q).
\end{equation}
In the other word,\ equation on $C(M)$ is equivalent to the equation on $\mathbf{R}\times M$ after rescaling of the metric.\ So we can only consider the instanton equation
$$\bar{\ast}F_{\textbf{A}}+\bar{\ast}Q_{\bar{Z}}\wedge F_{\textbf{A}}=0$$
on the cone $C(M)$ over $M$.\ Since
$$\bar{\ast}_{Z}Q_{\bar{Z}}=e^{(n-3)}\ast{Q_{Z}}=e^{(n-3)t}(\ast_{M}P+dt\wedge\ast_{M}Q),$$
by direct calculate,\  $\bar{\ast}_{Z}Q_{\bar{Z}}$ is closed.\ This implies that the instantons also satisfy the pure Yang-Mills equations with respect to the metric $g_{\bar{Z}}$.
\begin{prop}
Let $\textbf{A}$ be a instanton on $Z=\mathbf{R}\times M$,\ here $M$ is a closed Riemannian $n$-manifold,\ $n\geq 4$,\ which admits a smooth $3$-form $P$ and a smooth $4$-form $Q$ those satisfy equations (\ref{1.2}) and (\ref{1.3}).\ Then the connection $A$ is a Yang-Mills connection on $C(M)$.
\end{prop}
After rescaling of the metric,\
$$F_{\textbf{A}}\wedge\bar{\ast}F_{\textbf{A}}=e^{(n-3)t}F_{\textbf{A}}\wedge\ast F_{\textbf{A}}.$$
The curvature $F_{\textbf{A}}$ is $L^{2}$-bounded over $Z$.\ We shall prove that the curvature $F_{\textbf{A}}$ is still $L^{2}$-bounded over $C(M)$ by the following lemma.
\begin{lem}
Let $\textbf{A}$ be a instanton on $Z=\mathbf{R}\times M$ with $L^{2}$-bounded curvature $F_{\textbf{A}}$,\ i.e.
$$\int_{\mathbf{R}\times M}\langle F_{\textbf{A}}\wedge\ast F_{\textbf{A}}\rangle<+\infty,$$\
here $M$ is a closed Riemannian n-manifold,\ $n\geq 4$,\ which admits a 3-form $P$ and a 4-form $Q$ those satisfy equations (\ref{1.2}) and (\ref{1.3}).\ Then
$$\int_{\mathbf{R}\times M}\langle F_{\textbf{A}}\wedge\bar{\ast}F_{\textbf{A}}\rangle<+\infty.$$
\end{lem}
\begin{proof}
From theorem \ref{3.8},\ we have
$$L_{n-3}(T)=\int_{T}^{\infty}e^{(n-3)t}\|F_{\textbf{A}}\|_{L^{2}(M)}^{2}dt\leq Ce^{-(n-3)T}$$
Then for any constant $T\in[0,\infty)$,\ we have
\begin{equation}\nonumber
\begin{split}
\int_{\mathbf{R}\times M}\langle F_{\textbf{A}}\wedge\bar{\ast}F_{\textbf{A}}\rangle&=\int_{\mathbf{R}\times M}e^{(n-3)t}\langle F_{\textbf{A}}\wedge\ast F_{\textbf{A}}\rangle\\
&=\int_{(-\infty,T]\times M}+\int_{[T,\infty)\times M}e^{(n-3)t}\langle F_{\textbf{A}}\wedge\ast F_{\textbf{A}}\rangle\\
&\leq e^{(n-3)T}\int_{\mathbf{R}\times M}\langle F_{\textbf{A}}\wedge\ast F_{\textbf{A}}\rangle+L_{n-3}(T)<+\infty\\
\end{split}
\end{equation}
\end{proof}
We only consider instantons with $L^{2}$-bounded curvature on the cone of $M$.\ In the next section,\ we will give a vanishing theorem for Yang-Mills connection with finite energy on the cone of $M$.

\section{Vanishing Theorem for Yang-Mills}

In this section,\ notations may be different from the previous sections.\ We use the conformal technique to give the vanishing theorem for Yang-Mills connection on the cone of $M$.\

Let $M$ be a Riemannian $n+1$-manifold.\ Suppose $X\in\Gamma(TM) $ is a conformal vector field on $(M,g)$,\ namely,
$$\L_{X}g=2fg$$
where $f\in C^{\infty}(M)$.\ Here $\L_{X}$ denotes the Lie derivative with respect to $X$.

The vector field $X$ generates a family of local conformal diffeomorphism.
$$ F_{t}=exp(tX):\ M\rightarrow M$$
This family of local conformal diffeomorphism can induce a bundle automorphism,\ $\tilde{F}_{t}$,\ of the principal bundle $P$.\ Such a lift is readily obtained from a connection on $P$ by setting $\widetilde{F}_{t}=\exp(t\widetilde{X})$ where $\widetilde{X}$ is the horizontal lift of $X$ on $P$.\ If $A$ is the connection form we have $i_{\widetilde{X}}A=0$ since $\widetilde{X}$ is horizontal.\ Thus the Lie derivative of $A$ is can be expressed in terms of the curvature $F_{A}$:\ $\L_{\widetilde{X}}A=i_{\widetilde{X}}dA+di_{\widetilde{X}}A=i_{\widetilde{X}}(F_{A}-\frac{1}{2}[A,A])=i_{X}F_{A}$.\ And hence $\widetilde{F}^{\ast}_{t}A=A+ti_{X}F_{A}+o(t^{2}).$\ One can see the detailed process in \cite{T.Parker}.

We will consider the variation of the Yang-Mills functional under the family of diffeomorphism.
$$YM(A,g)=\int_{M}|F_{A}|^{2}dVol_{g}$$
where $dVol_{g}=\sqrt{detg}dx$ is the volume form of $M$.
$$\la:=|F_{A}|^{2}dVol_{g}$$
is an $n$-form on $M$.\ For any $\eta\in C^{\infty}_{0}(M)$
\begin{equation}\nonumber
\begin{split}
0&=\int_{M}d[(i_{X}\la)\eta]=\int_{M}\eta d(i_{X}\la)+\int_{M}d\eta\wedge i_{X}\la\\
&=\int_{M}\eta \L_{X}\la+\int_{M}d\eta\wedge i_{X}\la\\
\end{split}
\end{equation}
that is
\begin{equation}\label{4.1}
\int_{M}\eta \L_{X}\la=-\int_{M}d\eta\wedge i_{X}\la
\end{equation}
where $i_{X}$ stands for the inner product with the vector $X$.\ Now,\ let us compute $\L_{X}\la$.
\begin{lem}
Let $\la=Tr(F_{A}\wedge\ast F_{A})$ and $X$ be a smooth vector field on $M$ satisfying $\L_{X}g=2fg$,\ then
$$\L_{X}\la=(n-3)f\la+2Tr\big{(}d_{A}(i_{X}F_{A})\wedge\ast F_{A}\big{)}$$
\end{lem}
\begin{proof}
In local coordinates $\{x^{i}\}_{i=1}^{n}$,\ the n-form $\la$
$$\la=Tr(F_{A}\wedge\ast F_{A})=\sum g^{ij}g^{kl}trF_{ik}F_{jl}\sqrt{detg}dx^{1}\wedge\ldots\wedge dx^{n}$$
is conformal of weight $n-3$,\ i.e.\ $\la(A,e^{2f}g)=e^{(n-3)f}\la(A,g)$ for any $f\in C^{\infty}.$
The vector field $X$ satisfies $\L_{X}g=2fg$,\ so
$$F_{t}^{\ast}g=exp(2\int_{0}^{t}F^{\ast}_{s}f)g.$$
Since $\la$ is conformal with wight $n-3$,
\begin{equation}\nonumber
\begin{split}
(F_{t}^{\ast}\la)(A,g)&=\la(\tilde{F}_{t}^{\ast}A,F_{t}^{\ast}g)=exp\big{(}(n-3)\int_{0}^{t}F_{t}^{\ast}f\big{)}\cdot\la(\tilde{F}_{t}^{\ast}A,g)\\
&=\big{(}1+(n-3)tf+o(t^{2})\big{)}\\
&\times Tr\big{(}F_{A}\wedge\ast F_{A}+2td_{A}(i_{X}F_{A})\wedge\ast F_{A}+o(t^{2})\big{)}\\
&=Tr(F_{A}\wedge\ast F_{A}+t(n-3)fF_{A}\wedge\ast F_{A})\\
&+2tTr(d_{A}(i_{X}F_{A})\wedge\ast F_{A})+o(t^{2})\\
\end{split}
\end{equation}
where we used the fact $F_{t}^{\ast}f=f+o(t)$ and $\tilde{F}_{t}^{\ast}A=A+ti_{X}F_{A}+o(t^{2})$.
By the definition of Lie derivative
\begin{equation}\label{4.2}
\L_{X}\la=\frac{d}{dt}(F_{t}^{\ast}\la)|_{t=0}=(n-3)f\la+2Tr\big{(}d_{A}(i_{X}F_{A})\wedge\ast F_{A}\big{).}
\end{equation}
\end{proof}
We consider $M=\mathbf{R}\times N$ with metric
$$g_{M}=e^{2t}(dt^{2}+g_{N})$$
where $N$ is a compact Riemannian $n$-manifold,\ $n\geq4$,\ with metric $g_{N}$.\ Then the vector field $X=\frac{\pa}{\pa t} $ satisfies
$$\L_{X}g_{M}=X\cdot e^{2t}(dt^{2}+g_{N})+e^{2t}(\L_{X}dt^{2})=2g_{M},$$
and in this case,\ $f=1$.
\begin{thm}
Let $(M,g_{M})$ be a Riemannian manifold as above.\ Let $A$ be a Yang-Mills connection with $L^{2}$-bounded curvature $F_{A}$,\ i.e.
$$\int_{M}|F_{A}|^{2}<+\infty$$
over $M$.\ Then $A$ must be a flat connection.
\end{thm}
\begin{proof}
From (\ref{4.1}) and (\ref{4.2}),\ we have
\begin{equation}
\begin{split}
\int_{M}\eta(n-3)\la&=-\int_{M}d\eta\wedge i_{X}\la-2\int_{M}\eta Tr\big{(}d_{A}(i_{X}F_{A})\wedge\ast F_{A}\big{)}\\
&=-\int_{M}d\eta\wedge i_{X}\la-2\int_{M}Tr\big{(}d_{A}(\eta i_{X}F_{A})\wedge\ast F_{A}\big{)}\\
&+2\int_{M}Tr\big{(}(d\eta\wedge i_{X}F_{A})\wedge\ast F_{A}\big{)}\\
&\leq3\int_{M}|d\eta|\cdot|X|\cdot\la
\end{split}
\end{equation}
The second term in the second line vanishes since $A$ is a Yang-Mills connection.

We choose the cut-off function with $\eta(t)=1$ on the interval $|t|\leq T $,\ $\eta(t)=0$ on the interval $|t|\geq 2T $,\ and $|d\eta|\leq 2T^{-1}$.\ Then $d\eta$ has support in $ T\leq|t|\leq 2T $ and $|X(t)|=1,$
$$\int_{M}\eta(n-3)\la\leq\frac{6}{T}\int_{\{T\leq|t|\leq2T\}\times N}\la.$$
Letting $T\rightarrow\infty$ we get
$$\int_{M}(n-3)\la=0,$$
Then $F_{A}=0$.
\end{proof}

\subsection*{Acknowledgment}
I would like to thank O.~Lechtenfeld for kind comments regarding this and it companion article \cite{BILL,HILP} and S.K.~Donaldson for helpful comments regarding his article \cite{Donaldson},\ and H.N.~S\'{a} Earp for helpful comments regarding his article \cite{Earp,Earp2}.\ This work is partially supported by
Wu Wen-Tsun Key Laboratory of Mathematics of Chinese Academy of Sciences at USTC.


\bigskip
\footnotesize

\begin{thebibliography}{SK}

\bibitem{BILL}
Bauer,~I., Ivanova,~T.A., Lechtenfeld,~O. and Lubbe,~F., {\it Yang-Mills instantons and dyons on homogeneous $G_{2}$-manifolds.}
{JHEP.} {\bf2010}(10), 1--27 (2010)
\bibitem{BKLS}
Blumenhagen,~R., K\"{o}rs,~B., L\"{u}st,~D. and Stieberger,~S., {\it Four-dimensional string compactifications with D-branes, orientifolds and fluxes.} {Phys.Rept.} {\bf1} (445), 1--193 (2007)
\bibitem{BL}
Bourguignon,~J. P., Lawson,~H.B., {\it Stability and isolation phenomena for Yang-Mills fields,} {Comm. Math. Phys.} {\bf 79}(2), 189--230 (1981)
\bibitem{Bryant}
Bryant,~R.L., {\it Some remarks on $G_2$-structures,}
{arXiv preprint math/0305124,} (2003)
\bibitem{RRC}
Carri{\'o}n,~R.R., {\it A generalization of the notion of instanton,}
{Diff.Geom.Appl.} {\bf 8}(1), 1--20 (1998)
\bibitem{CDFJ}
Corrigan,~E.~,Devchand,~C.,Fairlie,~D.B., Nuyts.~J., {\it First order equations for gauge fields in spaces of dimension great than four,}
{Nucl.Phys.B,} {\bf 214}(3), 452--464 (1983)
\bibitem{Donaldson}
Donaldson,~S.K., {\it Floer homology groups in Yang-Mills theory},
{Cambridge University Press,} (2002)
\bibitem{DT}
Donaldson,~S.K., Thomas~R.P., {\it  Gauge theory in higher dimensions,}
{The Geometric Universe, Oxford,} 31--47 (1998)
\bibitem{DS}
Donaldson,~S.K., Segal.~E., {\it Gauge theory in higher dimensions, II.}
{arXiv:0902.3239,} (2009)
\bibitem{DK}
Douglas,~M.R., Kachru,~S., {\it Flux compactification.}
{Rev.Mod.Phys.} {\bf79}(2):733 (2007)
\bibitem{Gr}
Gra{\~{n}}a,~M.,{\it Flux compactifications in string theory: A comprehensive review.}
{Phys.Rept.} {\bf423}(3), 91--158 (2006)
\bibitem{GSW}
Green,~M.B., Schwarz,~J.H. and Witten,~E., {\it Supperstring theory,}
{Cambridge University Press,} (1987)
\bibitem{HILP}
Harland,~D., Ivanova,~T.A., Lechtenfeld,~O. and Popov,~A.D., {\it Yang-Mills flows on nearly K\"{a}hler manifolds and $G_{2}$-instantons.}
{Comm.Math.Phys.}  {\bf 300}(1), 185--204 (2010)
\bibitem{HN}
Harland,~D., N{\"o}lle~.C, {\it Instantons and Killing spinors,}
{JHEP.} {\bf3}, 1--38 (2012)
\bibitem{ID}
Ivanova.~T.A., Popov,~A.D., {\it Instantons on special holonomy manifolds,}
{Phys.Rev.D} {\bf 85}(10) (2012)
\bibitem{ILPR}
Ivanova,~T.A., Lechtenfeld,~O., Popov,~A.D. and Rahn,~T., {\it Instantons and Yang-Mills flows on coset spaces.}
{Lett. Math. Phys.} {\bf 89} (3), 231--247 (2009)
\bibitem{li1984p}
Li,~P., Schoen,~R., {\it $L^p$ and mean value properties of subharmonic functions on Riemannian manifolds},
{Acta Mathematica}, {\bf 153}(1), 279--301 (1984)
\bibitem{morrey1966multiple}
Morrey,~C.B., \textit{{{Multiple integrals in the calculus of variations}}},
Springer,(1966)
\bibitem{T.Parker}
Parker,~T., {\it Conformal fields and stability},
{Math.Z.} {\bf 185}(3) 305--319 (1984)
\bibitem{Earp}
S\'{a} Earp,~H.N., {\it Instantons on $G_2$-manifolds,}
{London Ph.D Thesis,} (2009)
\bibitem{Earp2}
S\'{a} Earp,~H.N., {\it Generalised Chern-Simons Theory and $G_{2}$-Instantons over Associative Fibrations.}
SIGMA, {\bf10}:083 (2014)
\bibitem{G.Tian}
Tian,~G., {\it Gauge theory and calibrated geometry, I.}
{Ann.Math.} {\bf 151}(1), 193--268 (2000)
\bibitem{RSW}
Ward,~R.S., {\it Completely solvable gauge field equations in dimension great than four,}
{Nucl.Phys.B} {\bf 236}(2), 381--396 (1984)


\normalsize
\baselineskip=17pt

\end{thebibliography}
\end{document}